\documentclass[12pt]{amsart}

\usepackage{amssymb}
\usepackage{dsfont}
\usepackage{enumerate}
\usepackage{xcolor}

\newtheorem{lemma}[equation]{Lemma}
\newtheorem{thm}[equation]{Theorem}
\newtheorem{prop}[equation]{Proposition}

\newtheorem{conj}[equation]{Conjecture}

\DeclareMathOperator{\Tr}{Tr}

\title{On approximate commutativity of spaces of matrices}
\author{Matja\v{z} Omladi\v{c}, Heydar Radjavi, Klemen \v Sivic}
\begin{document}

\begin{abstract}
    The maximal dimension of commutative subspaces of $M_n(\mathbb{C})$ is known.  So is the structure of such a subspace when the maximal dimension is achieved. We consider extensions of these results and ask the following natural questions: If $V$ is a subspace of $M_n(\mathbb{C})$ and $k$ is an integer less than $n$,  such that for every pair $A$ and $B$ of members of $V$, the rank of the commutator $AB - BA$ is at most $k$, then how large can the dimension of $V$ be?  If this maximum is achieved, can we determine the structure of $V$?  We answer the first question.  We also propose a conjecture on the second question which implies, in particular, that such a subspace $V$ has to be an algebra, just as in the known case of $k = 0$.  We prove the proposed structure of $V$ if it is already assumed to be an algebra.

\end{abstract}

\thanks{MO\&K\v{S} acknowledge financial support from the Slovenian Research Agency (research core funding No. P1-0222). K\v{S} also acknowledges grant No. N1-0103 from the Slovenian Research Agency.\\This paper was accepted in the Linear Algebra and its applications}

\maketitle

\section{ Introduction }

Let $V$ be a subspace of $M_n(\mathbb{C})$, the space of $n\times n$ complex matrices. Can the structure of $V$ be determined if it is ``almost commutative'' in some sense? 
We are interested in the following special case of this question: Fix $k< n$. What can we say about $V$ if
\[
    \mathrm{rank}(AB-BA)\le k
\]
for all $A$ and $B$ in $V$? 
The case of $k=0$, or commutative $V$, has been studied {before}. For example, it is shown in \cite{Schur} that the maximal dimension of such a subspace $V$ is $rs+1$ with $\displaystyle r=\left\lfloor \dfrac{n}{2} \right\rfloor$ and $\displaystyle s=\left\lfloor\dfrac{n+1}{2} \right\rfloor$. Furthermore, if this dimension is achieved, then $V$ is actually an algebra, which is similar to $V_0$ or $V_0^T$, where
\[
    V_0=\left\{\begin{pmatrix}
          \alpha I_r & X \\
          0 & \alpha I_s
        \end{pmatrix}\,;\,X\in M_{r\times s}(\mathbb{C}),\alpha\in\mathbb{C}\right\},
\]
and where $M_{r\times s}(\mathbb{C})$ means the set of all $r$ by $s$ complex matrices (and $V_0^T$ is the set $\{A^T;A\in V_0\}$). {These are all the possible solutions for such a space in the case $n\ge4$. In the two remaining cases the dimension of a maximal commutative space equals $n$. If $n=2$ we have one more solution for the space, {namely,} the diagonal matrices. If $n=3$ we have three exceptional solutions: the diagonals, the direct sum of a 2-dimensional matrix algebra generated by {a rank-one nilpotent} with $\mathbb{C}$, and the algebra generated by a nilpotent of rank two. This list of exceptional cases is well-known and complete as {it is not hard to see.} }

For $k>0$ we prove that the maximal possible dimension of $V$ is
\[
    kn+\left\lfloor\dfrac{n-k}{2}\right\rfloor\left\lfloor\dfrac{n-k+1}{2} \right\rfloor+1,
\]
which generalizes the inequality given above. The main tool in achieving this result is the Borel fixed point theorem which enables us to reduce the problem to the spaces spanned by matrix units and diagonal matrices. We also conjecture that if this maximum is achieved {and $n-k$ is different from $2$ or $3$}, then $V$ is actually an algebra similar to  $V_k$ or $V_k^T$, where $V_k$ consists of matrices of the block form
\[
    \begin{pmatrix}
          X & Y & Z \\
          0 &\alpha I_{r_k} & W \\
          0 & 0 & \alpha I_{s_k}
        \end{pmatrix},
\]
where $r_k$ and $s_k$ are equal to $\left\lfloor\dfrac{n-k}{2}\right\rfloor$ and $\left\lfloor \dfrac{n-k+1}{2} \right\rfloor$ in either order, and $X,Y,Z,$ and $W$ are arbitrary blocks.

We have only been able to prove this general conjecture in the two extreme cases of $k=1$ and $k=n-1$. But if we assume beforehand that $V$ is an algebra, then we can show that the structure result holds. We also prove a few other related results. In particular, we show that every linear space of matrices with {rank-one} commutators
 is simultaneously triangularizable.

\section{ Some essential results }

In this section we treat the two cases of $k=1$ and $k=n-1$. The treatment is purely linear algebraic.

\begin{lemma}
A linear space of $n\times n$ matrices of dimension at least $n^2-n+1$ contains a matrix with $n$ distinct eigenvalues.
\end{lemma}

This lemma follows from \cite{LoRa}; see also \cite{OS}.



\begin{thm}\label{thm:2}
Let $n>2$ and $V$ be a linear space of $n\times n$ matrices such that the matrix $[A,B]=AB-BA$ is not invertible for any $A,B\in V$. Then $\dim V\le n^2-n+1$. Moreover, each such space of dimension $n^2-n+1$ is either similar to the space of all matrices whose first $n-1$ entries in the last row are zero or to the space of all matrices whose last $n-1$ entries in the first column are zero.
\end{thm}

\begin{proof}
We prove the first part of the theorem by contradiction. Assume that $V$ is a linear space of $n\times n$ matrices such that $[A,B]$ is not invertible for any $A,B\in V$ and suppose that $\dim V\ge n^2-n+2$. By passing to a subspace we may assume that $\dim V=n^2-n+2$. By the above lemma there exists a matrix $D\in V$ that has $n$ distinct eigenvalues. Without any loss of generality we may assume that $D$ is diagonal.

Now we define the space $W=\{[D,X];X\in V\}$. Clearly $\dim W\ge n^2-2n+2$ and by the assumption no element of $W$ is invertible. We want to show that actually  $\dim W= n^2-2n+2$. Assume to the contrary that $\dim W\ge n^2-2n+3$. Then all the matrices in $W$ have common kernel or common cokernel by \cite[Theorem 3]{FLH}. We only treat the case in which they have common kernel since the other case is proved in a similar way by symmetry. Let ${w}$ be a nonzero vector in the common kernel. Then $\Tr(X{w} {x}^T)=0$ for any ${x} \in \mathbb{C}^n$ and any $X\in W$. Additionally, the space $W$ contains only matrices with zero diagonals, therefore $\Tr(XE_{ii})=0$ for each $i=1,\ldots ,n$ and each $X\in W$. So, all matrices of the form ${w} {x}^T+D'$ where ${x}\in \mathbb{C}^n$ and $D'$ is diagonal belong to the space $W^{\perp}$ whose dimension is at most $2n-3$. (Here, $W^{\perp}=\{A\;;\;\Tr(AX)=0,\ \mbox{for every}\ X\in W\}$.) However, this is clearly not possible, as a nonzero matrix ${w}{x}^T$ can be diagonal only if ${w}$ and ${x}$ are multiples the same standard basis vector of $\mathbb{C}^n$.

Consequently, the space $W$, whose dimension is equal to $n^2-2n+2$, is the image of the linear map $V\to M_n(\mathbb{C})$ defined by $X\mapsto [D,X]$, and the kernel of this map, which is the centralizer of $D$ in $V$, is $n$-dimensional. Since $D$ has $n$ distinct eigenvalues, it follows that $V$ contains all diagonal matrices. Moreover, the same proof shows that for any matrix $A\in V$ that has $n$ distinct eigenvalues all the powers of $A$ also lie in $V$. Now let $X\in V$ be any matrix. Since the set of the matrices with $n$ distinct eigenvalues is open, there exists a sufficiently small $\varepsilon >0$ such that $D+\lambda X$ has $n$ distinct eigenvalues for each $\lambda <\varepsilon$. Consequently, $(D+\lambda X)^2\in V$ for each $\lambda <\varepsilon$, and in particular, $X^2\in V$. This shows that $V$ is a Jordan algebra.

Since $\dim V=n^2-n+2$, the Jordan algebra $V$ is irreducible, hence it is simple by \cite[Proposition 3.1]{GKOR}. Since $V$ contains an element with $n$ distinct eigenvalues, the rank of the Jordan algebra $V$ is $n$. However, by the classification of simple Jordan algebras \cite[Chapter VIII.5]{FaKo}, the dimension of $V$ is then $\frac{n(n+1)}{2}$, $n^2$ or $n(2n-1)$, or 27 if $n=3$. None of these numbers can be equal to $n^2-n+2$, 
which concludes the proof of the first part of the theorem.

For the second part first observe that the space of all matrices with first $n-1$ entries of the last row equal to zero and the space of all matrices with last $n-1$ entries of the first column equal to zero clearly satisfy the condition of the theorem. Suppose now that $V$ is an arbitrary linear space of $n\times n$ matrices of dimension $n^2-n+1$ such that $[A,B]$ is not invertible for any $A,B\in V$. As above, the space $V$ contains a matrix $D$ with $n$ distinct eigenvalues by the previous lemma, and we may assume that it is diagonal. If we let $W=\{[D,X];X\in V\}$, then $\dim W\ge n^2-2n+1$.

We now show that $\dim W=n^2-2n+1$. If $\dim W\ge n^2-2n+2$, then \cite[Theorem]{AL} implies that the matrices from $W$ have either common kernel, or common cokernel, or else there exist a 2-dimensional subspace $U_1\subseteq \mathbb{C}^n$ and an $(n-1)$-dimensional subspace $U_2\subseteq \mathbb{C}^n$ such that either ${x}^TA{y}=0$ for all ${x}\in U_1$, ${y}\in U_2$ and $A\in W$ or ${y}^TA{x}=0$ for all ${x}\in U_1$, ${y}\in U_2$ and $A\in W$. With an argument similar to the one given above we show that the first two cases are not possible, as a vector space of matrices with zero diagonal and common kernel or common cokernel has codimension at least $2n-1$. However, the last two cases are not possible either, since the space of all matrices $A\in M_n(\mathbb{C})$ satisfying, for example, ${x}^\mathrm{T} A{y}=0$ for all ${x}\in U_1$ and ${y}\in U_2$, has dimension $n^2-2n+2$ and it contains matrices with nonzero diagonal, so it cannot contain $W$.

So, $\dim W=n^2-2n+1$, and the space $V$ contains all diagonal matrices, and the above argument shows that $V$ is a Jordan algebra. If it is irreducible, then it is simple, and as above we get a contradiction by the classification of simple Jordan algebras. Hence $V$ is reducible. However, since $\dim V=n^2-n+1$, this is possible only if the matrices from $V$ have a common left eigenvector or a common right eigenvector, which concludes the proof of the theorem.
\end{proof}

\noindent\textbf{Remark.} This theorem is valid also in the case $n=2$, which will be shown in Theorem \ref{maksrankone}.

\begin{lemma}\label{rankone}
Let $V$ be a linear space of $n\times n$ matrices such that $\mathrm{rank}(AB-BA)\le 1$ for any $A,B\in V$. Then there exists ${x}\in \mathbb{C}^n$ such that either $\{[A,B]; A,B\in V\}\subseteq \{{x}{y}^T;{y}\in \mathbb{C}^n\}$ or $\{[A,B]; A,B\in V\}\subseteq \{{y}{x}^T;{y}\in \mathbb{C}^n\}$.
\end{lemma}

\begin{proof}
If all the commutators of the elements of $V$ are zero, then the lemma clearly holds. Therefore we assume that there {exist two matrices in $V$ with rank-one commutator.} %
We first show that for each $A\in V$ there exists ${x}\in \mathbb{C}^n$ such that  either $\{[A,B]; B\in V\}\subseteq \{{x}{y}^T;{y}\in \mathbb{C}^n\}$ or $\{[A,B];B\in V\}\subseteq \{{y}{x}^T;{y}\in \mathbb{C}^n\}$. Choose an $A\in V$ and denote $[A,B]={x}{y}^T$ for some $B\in V$ and some nonzero ${x},{y}\in \mathbb{C}^n$. Towards a contradiction assume that there exist matrices $C,D\in V$ such that $[A,C]={z}{w}^T$ and $[A,D]={u}{v}^T$ where ${u},{v},{w},{z}$ are nonzero, ${z}$ is not parallel to ${x}$, and ${v}$ is not parallel to ${y}$. Since $\mathrm{rank}[A,B+\lambda C]\le 1$ and $\mathrm{rank}[A,B+\lambda D]\le 1$ for all $\lambda \in \mathbb{C}$, it follows that ${w}$ is parallel to ${y}$ and ${u}$ is parallel to ${x}$. Since ${w}$ and ${u}$ are nonzero, we may assume that ${w}={y}$ and ${u}={x}$. However, then $[A,C+\lambda D]={z}{y}^T+\lambda {x}{v}^T$ is of rank 2 for some $\lambda \ne 0$, which is in a contradiction with the {rank-one assumption}.

We now prove the lemma by contradiction. Assume that there are $A,B\in V$ such that $[A,B]={x}{y}^T$ for some nonzero ${x},{y}\in \mathbb{C}^n$, and there are $C,D\in V$ such that $[C,D]={z}{w}^T$ where ${z}$ is not parallel to ${x}$ and ${w}$ is not parallel to ${y}$.
Using the above, we may assume with no loss that $\{[A,X];X\in V\}\subseteq \{{x}{u}^T;{u}\in \mathbb{C}^n\}$, since otherwise we would replace $V$ by the linear space of its transposes.
Since the commutators $[A+\lambda D,C]$ and $[A+\lambda C,D]$ have rank at most 1 for all $\lambda \in \mathbb{C}$, it follows that $[A,C]=\alpha {x}{w}^T$ and $[A,D]=\beta {x}{w}^T$ for some $\alpha ,\beta \in \mathbb{C}$. The commutator $[C,D]$ does not change (up to a sign) if we exchange $C$ and $D$ or add some multiple of $C$ to $D$. Therefore we lose no generality if we assume that $[A,D]=0$. Since the commutators $[A+\lambda C,B]$ and $[C,D+\lambda B]$ also have rank at most 1 for all $\lambda \in \mathbb{C}$, it follows that $[C,B]=\gamma {z}{y}^T$ for some $\gamma \in \mathbb{C}$ or $[C,B]=\delta {x}{w}^T$ for some $\delta \in \mathbb{C}$. However, in both cases the rank of the commutator $[A+\lambda C,B+\lambda D]$ is 2 for all but finitely many scalars $\lambda \in \mathbb{C}$, which is in a contradiction with the rank one assumption.
\end{proof}

\begin{thm}\label{maksrankone}
Let $n>1$ and let $V$ be a linear space of $n\times n$ matrices such that $\mathrm{rank}(AB-BA)\le 1$ for every $A,B\in V$. Then $\dim V\le \lfloor \frac{(n-1)^2}{4}\rfloor +n+1$. Moreover, if $V$ is such a space of maximal dimension, then either $V$ or $V^T$ is similar to the space of all matrices of the form
\[
        \begin{pmatrix}
        \lambda&{a}^T\\0& { C}
        \end{pmatrix}
\]
where $\lambda\in\mathbb{C}$, ${a}\in \mathbb{C}^{n-1}$ are arbitrary. { If $n>4$, then $C$} is of the form
\[
     { C}=        \begin{pmatrix}
                \mu I_k& { D}\\0&\mu I_{n{-1}-k}
              \end{pmatrix}
\]
for some $\mu \in \mathbb{C}$ and some $ { D}\in M_{k\times (n-1-k)}$ where $k\in \{\lfloor \frac{n-1}{2}\rfloor, \lceil \frac{n-1}{2}\rceil\}.$  { For $n\le4$ the possibilities for $C$ are as follows:
\begin{enumerate}[(a)]
  \item For $n=4$ we have three additional options: 
      {The space of all possible matrices $C$ is either the space of all diagonal matrices, or a direct sum of a two-dimensional algebra generated by a nilpotent of rank one with $\mathbb{C}$, or the algebra of $3\times 3$ matrices generated by a nilpotent of rank two.}
  \item For $n=3$ there is only one additional option, namely that $C$ is taken from the algebra of all $2\times2$ diagonals.
  \item For $n=2$ the only option is that $C$ be an arbitrary scalar.
\end{enumerate}}
\end{thm}

\begin{proof}
By Lemma \ref{rankone}, after a simultaneous similarity, and by going to the transposes, if necessary, we may assume that all the commutators of the elements of $V$ have nonzero entries only in the first row. Assume that $\dim V\ge \lfloor \frac{(n-1)^2}{4}\rfloor +n+1$.

We will show that $V$ always contains a nonzero matrix of the form ${e_1}{u}^T$. Towards a contradiction assume that $V$ contains no nonzero matrix of the form ${e_1}{x}^T$. Let $W=\{A\in V;A{e_1}=0\}$. Then the matrices of $W$ are of the form $\begin{pmatrix}
                                                           0&{a}^T\\0&A'
                                                         \end{pmatrix}$ and $\dim W\ge  \lfloor \frac{(n-1)^2}{4}\rfloor +1$. The projection $\pi \colon W\to M_{n-1}(\mathbb{C})$ that sends $\begin{pmatrix}
                                                                 0&{a}^T\\0&A'
                                                                \end{pmatrix}$ to $A'$ is injective, since by the assumption $V$ does not contain nonzero matrices of the form $\begin{pmatrix}
                                                                   0&{a}^T\\0&0
                                                                 \end{pmatrix}$. Consequently, $\dim \mathrm{im}\, \pi =\dim W\ge  \lfloor \frac{(n-1)^2}{4}\rfloor +1$ (here $\mathrm{im}\, \pi$ means image or range of $\pi$). However, since the commutators of elements of $V$ (and therefore of $W$) have nonzero entries only in the first row, the elements of $\mathrm{im}\, \pi$ commute.
Schur's theorem \cite{Schur} then implies that $\dim \mathrm{im}\, \pi =\lfloor \frac{(n-1)^2}{4}\rfloor +1$ and that $\mathrm{im}\, \pi$ is similar to the space of all matrices of the form $\begin{pmatrix}
            \lambda I_k&X\\0&\lambda I_{n-k-1}
          \end{pmatrix}$ where $\lambda \in \mathbb{C}$ and $X$ is an arbitrary $k\times (n-1-k)$-matrix where $k\in \{\lfloor \frac{n-1}{2}\rfloor, \lceil \frac{n-1}{2}\rceil\}.$ In particular, the space $W\subseteq V$ contains a matrix of the form $A=\begin{pmatrix}
               0&{a}^T\\0&I
             \end{pmatrix}$. {(Here, we need to treat the cases $n<5$ meaning $n-1<4$ separately. Indeed, in all exceptional cases we also find $I$ in $\mathrm{im}\, \pi$.)} Let $B=\begin{pmatrix}
                                      \beta&{b_1}^T\\{b_2}&B'
                                    \end{pmatrix}\in V$ be arbitrary. The commutator $[A,B]$ has nonzero entries only in the first row, which implies that ${b_2}=0$. Since $B\in V$ was arbitrary, it follows that ${e_1}$ is an eigenvector of every matrix from $V$. However, then $\dim W\ge \lfloor \frac{(n-1)^2}{4}\rfloor +n$ and $\dim \mathrm{im}\, \pi =\dim W\ge \lfloor \frac{(n-1)^2}{4}\rfloor +n$, which is a contradiction.

We have thus proved that $V$ contains a nonzero matrix of the form ${e_1}{u}^T$. Let $A\in V$ be arbitrary. Then $[A,{e_1}{u}^T]=A{e_1}{u}^T- {e_1}{u}^TA={e_1}{x}^T$ for some ${x}\in \mathbb{C}^n$. Choose a vector ${v}\in \mathbb{C}^n$ with ${u}^T{v}\ne 0$ to get $A{e_1}{u}^T{v} ={e_1}({x}^T+{u}^TA){v}$ implying that ${e_1}$ is an eigenvector for each matrix of $V$. So, the matrices of $V$ are of the form $\begin{pmatrix}
                                   \alpha&{a}^T\\0&A'
                                 \end{pmatrix}.$ Since the commutator of any two matrices of the kind has nonzero entries only in the first row, the lower right $(n-1)\times (n-1)$ corners of the matrices of $V$ commute. Moreover, they form a linear space of dimension at least $\lfloor \frac{(n-1)^2}{4}\rfloor +1$. Schur's theorem \cite{Schur} once again implies that the dimension is exactly $\lfloor \frac{(n-1)^2}{4}\rfloor +1$ and that the space of lower right $(n-1)\times (n-1)$ corners of the matrices of $V$ is  { either} similar  to the space of all matrices of the form $\begin{pmatrix}
                                         \lambda I_k&X\\0&\lambda I_{n-k-1}
                                       \end{pmatrix}$ where $\lambda \in \mathbb{C}$ and $X$ is an arbitrary $k\times (n-1-k)$-matrix where $k\in \{\lfloor \frac{n-1}{2}\rfloor, \lceil \frac{n-1}{2}\rceil\}$  { or it has one of the exceptional forms described in the introduction for small $n$.} Since $\dim V\ge \lfloor \frac{(n-1)^2}{4}\rfloor +n+1$ and ${e_1}$ is an eigenvector of every matrix of $V$, it follows that $\dim V=\lfloor \frac{(n-1)^2}{4}\rfloor +n+1$, and $V$ has the required form.
\end{proof}

Theorems \ref{thm:2} and  \ref{maksrankone} suggest the following conjecture whose first part (i.e., the statement on the dimension of $V$) will be proved in Section \ref{sec:main}.

\begin{conj}\label{conj}
Let $k$ be an arbitrary integer with $0\le k< n$, and let $V$ be a linear space of $n\times n$ matrices such that $[A,B]$ is of rank at most $k$ for each $A,B\in V$. Then $\dim V\le nk+\lfloor \frac{(n-k)^2}{4}\rfloor +1$. Moreover, in the case of equality $V$ or $V^T$ is similar to the space of all matrices of the form $\begin{pmatrix}
                               A&B\\0&C
                             \end{pmatrix}$ where $A\in M_k(\mathbb{C})$, $B\in M_{k\times (n-k)}(\mathbb{C})$, and $C\in M_{n-k}(\mathbb{C})$ is of the form $C=
\begin{pmatrix}
  \lambda I_l&D\\0&\lambda I_{n-k-l}
\end{pmatrix}$ with $D\in M_{l\times (n-k-l)}(\mathbb{C})$ and $l\in \left\{ \lfloor \frac{n-k}{2}\rfloor ,\lceil \frac{n-k}{2}\rceil\right\}$. { The additional possibilities for $C$ apply as in Theorem \ref{maksrankone} in the cases $n-k=3,2,$ respectively $1$.}
\end{conj}

Note that the conjecture holds for $k=0$ by \cite{Schur} and for $k=1$ and $k=n-1$ by the theorems above.

In Theorem \ref{maksrankone} we showed, in particular, that every linear space whose commutators are of rank at most one is simultaneously triangularizable if the space is of maximal possible dimension with this property. It turns out, somewhat surprisingly, that if we drop the assumption of maximal dimension, then we still get this particular conclusion.

The following theorem is an extension of a result of Laffey \cite{Laff}, discovered independently by Barth-Elencwajg \cite{BaEl}; for nonzero characteristics see Guralnick \cite{Gura}.

\begin{thm}
  Let $V$ be a linear space of $n\times n$ matrices and assume that $\mathrm{rank}(AB-BA)\le1$ for all $A$ and $B$ in $V$. Then $V$ is simultaneously triangularizable.
\end{thm}

\begin{proof}
  By Lemma \ref{rankone} and going to the set of transposes of the set $V$ if necessary we can assume that there exists a nonzero $x_0$ in $\mathbb{C}^n$ such that
  \[
    \{[A,B];A,B\in V\}\subseteq\{x_0y^T;y\in\mathbb{C}^n\}.
  \]
  Observe also that, by induction on $n$, it suffices to show that $V$ has a nontrivial invariant subspace. (We treat matrices as linear operators, as usual.) We will also assume with no loss of generality that $I\in V$.

  Pick a nonzero singular member $B_0$ of $V$ and let $\mathcal{M}$ be its kernel. Relative to the decomposition of $\mathbb{C}^n$ into direct sum of $\mathcal{M}$ with one of its complements the matrix of $B_0$ is
  \[
    B_0= \begin{pmatrix}
           0&Y_0\\0&T_0
         \end{pmatrix}
  \]
  and the matrix of a general member $A$ of $V$ is
  \[
    A= \begin{pmatrix}
         X&Y\\Z&T
       \end{pmatrix} .
  \]
  If $Z=0$ for every $A\in V$, we are done, because $\mathcal{M}$ is invariant under $V$.

  Now assume $Z\neq0$ for some $A$, then
  \[
    B_0A-AB_0= \begin{pmatrix}
                 Y_0Z&*\\T_0Z&*
               \end{pmatrix} .
  \]
  Since the first block column $C=\displaystyle \begin{pmatrix}
                                                  Y_0\\T_0
                                                \end{pmatrix} Z$ is not zero (because otherwise the kernel of $B_0$ would be larger than $\mathcal{M}$) we observe that the range of $C$ has to contain $x_0$, the vector determining the range of the rank-one operator $B_0A-AB_0$. Thus $x_0$ is also in the range $\mathcal{R}$ of $B_0$. But this implies that $\mathcal{R}$ is invariant for $V$, because for an arbitrary $A$ in $V$ we have $AB_0=B_0A+x_0y_A^T$, which implies that for every $x\in\mathbb{C}^n$,
  \[
    AB_0x=B_0Ax+(y_A^Tx)x_0\in\mathcal{R}.
  \]
\end{proof}

\section{The main result}\label{sec:main}

In this section we show that the inequality $\dim V\le nk+\left\lfloor \dfrac{(n-k)^2}{4} \right\rfloor +1$ from Conjecture \ref{conj} is true using the ideas of \cite{DKK} via {the Borel Fixed Point Theorem \cite[Theorem 10.4]{Bor}:} \emph{Let $X$ be a non-empty projective variety and $G$ a connected solvable algebraic group acting on it {via regular maps}. Then this action has a fixed point.} 

Fix an integer $k$, $0\le k\le n-1$, and $m\ge nk+\left\lfloor \dfrac{(n-k)^2}{4} \right\rfloor +1$. Let $V$ be a vector subspace of $M_n(\mathbb{C})$ with the following properties.
\begin{enumerate}[(A)]
  \item For all $A,B\in V$ we have $\mathrm{rank}\,[A,B]\le k$.
  \item The dimension of $V$ is $m$.
\end{enumerate}

 It is clear that $X$, the set of all vector spaces $V$ having these properties, is a subset of the Grassmannian variety $\mathrm{Gr}(m,n^2)$ of all $m$-dimensional subspaces of $\mathbb{C}^{n^2}$. Recall that $\mathrm{Gr}(m,n^2)$ is a projective variety in $\mathbb{P}(\wedge^m(\mathbb{C}^{n^2}))=\mathbb{P}^{{n^2\choose m}-1}$ via the Pl\"ucker embedding that sends {a vector space $V$ with a  basis $\{v_1,\ldots ,v_m\}$} to the class $[v_1\wedge \cdots \wedge v_m]\in \mathbb{P}(\wedge^m(\mathbb{C}^{n^2}))$. Note that this map is well-defined and that its image is indeed a subvariety of $\mathbb{P}(\wedge^m(\mathbb{C}^{n^2}))$ \cite[Chapter 6]{Har}. To prove that $X$ is a projective variety it therefore suffices to show that the condition (A) is closed in the Zariski topology.

Now consider $\mathbb{C}^{2n^2}$ as $\mathbb{C}^{n^2}\times \mathbb{C}^{n^2}$ and define the set \[Y=\{(W,V)\in \mathrm{Gr}(2m,2n^2)\times \mathrm{Gr}(m,n^2);W=V\times V\}.\]
Furthermore, let $i_1,i_2\colon \mathbb{C}^{n^2}\to \mathbb{C}^{2n^2}$ be embeddings defined by
$i_1(v)=(v,0)$ and $i_2(v)=(0,v)$, respectively. The maps $i_1$ and $i_2$ define embeddings
\[
\widetilde{i_1},\widetilde{i_2}\colon \mathrm{Gr}(m,n^2)\to \mathrm{Gr}(m,2n^2)\subseteq \mathbb{P}(\wedge^m(\mathbb{C}^{2n^2}))\quad\mbox{\cite[p.66]{Har}.}
\]
Let $V$ be an $m$-dimensional subspace of $\mathbb{C}^{n^2}$ with a basis $\mathcal{B}_V=\{v_1,\ldots ,v_m\}$ and let $W$ be a $2m$-dimensional subspace of $\mathbb{C}^{2n^2}$ with a basis $\{w_1,\ldots ,w_{2m}\}$. It is clear that $W=V\times V$ if and only if $i_1(\mathcal{B}_V)\cup i_2(\mathcal{B}_V)$ is a basis of $W$, which is equivalent to $\widetilde{i_1}([v_1\wedge \cdots \wedge v_n])\wedge \widetilde{i_2}([v_1\wedge \cdots \wedge v_n])=[w_1\wedge \cdots \wedge w_{2m}]$. This is a closed condition, which shows that $Y$ is a subvariety of $\mathrm{Gr}(2m,2n^2)\times \mathrm{Gr}(m,n^2)$.

Next, let $Z$ be the set of all pairs of matrices $(A,B)\in M_n(\mathbb{C})^2$ satisfying $\mathrm{rank}\,[A,B]\le k$. This is clearly an affine variety in $\mathbb{C}^{2n^2}$, since it is defined by vanishing $(k+1)\times (k+1)$ minors of $[A,B]$. This variety is homogeneous (i.e., if $(A,B)\in Z$ and $\alpha\in\mathbb{C}$ then $(\alpha A,\alpha B)\in Z$), so we may view it as a projective variety.
Let $$F_{2m}(Z)=\{W\in \mathrm{Gr}(2m,2n^2);W\subseteq Z\}.$$ By \cite[Example 6.19]{Har} $F_{2m}(Z)$ is a subvariety of $\mathrm{Gr}(2m,2n^2)$, called the Fano variety of $Z$. Consequently, $F_{2m}(Z)\times \mathrm{Gr}(m,n^2)$ is a subvariety of $\mathrm{Gr}(2m,2n^2)\times \mathrm{Gr}(m,n^2)$ and so is the intersection $(F_{2m}(Z)\times \mathrm{Gr}(m,n^2))\cap Y$.

Finally, observe that $X$ is the image of $(F_{2m}(Z)\times \mathrm{Gr}(m,n^2))\cap Y$ under the projection $\mathrm{Gr}(2m,2n^2)\times \mathrm{Gr}(m,n^2)\to \mathrm{Gr}(m,n^2)$. A projection along a projective space is a closed map in the Zariski topology, therefore the set $X$ is closed in $\mathrm{Gr}(m,n^2)$. We have proved the following lemma.

\begin{lemma}\label{lem:7}
Let $k$ be an integer, $0\le k\le n-1$, and $m\ge nk+\left\lfloor\dfrac{(n-k)^2}{4} \right\rfloor +1$. Then the set $X$ of all $m$-dimensional vector subspaces $V$ of $M_n(\mathbb{C})$ which satisfy the condition (A) is a projective variety.
\end{lemma}

Let $G$ be the solvable algebraic group of all invertible upper triangular matrices and define an action of $G$ on $\mathrm{Gr}(m,n^2)$ by $G\times \mathrm{Gr}(m,n^2)\longrightarrow \mathrm{Gr}(m,n^2), (P,V)\mapsto PVP^{-1}$. It is an easy task to verify that this action preserves the condition (A), so that $X$ is invariant under it. We can thus apply the Borel fixed point theorem to get the following lemma.

\begin{lemma}\label{borel}
  Let $0\le k \le n-1$ and $m\ge nk+\left\lfloor\dfrac{(n-k)^2}{4} \right\rfloor +1$. If there exists a vector subspace $W$ of $M_n(\mathbb{C})$ satisfying conditions \emph{(A)} and \emph{(B)}, then there exists a vector subspace $V$ of $M_n(\mathbb{C})$ satisfying conditions \emph{(A)} and \emph{(B)} and also such that $PAP^{-1} \in V$ for all $A\in V$ and every invertible upper triangular matrix $P$.
\end{lemma}

In the following lemma we need the notation $[n]=\{1,2,\ldots,n\}$.

\begin{lemma}\label{Eij} Let $V\subseteq M_n(\mathbb{C})$ be a vector space such that $PAP^{-1}\in V$ for all $A\in V$ and all invertible upper triangular $P\in M_n(\mathbb{C})$. 
  \begin{enumerate}[(a)]
    \item If for some $A\in V$ and $i,j\in[n]$ with $i<j$ we have $a_{ji}\ne0$, then $E_{ij}\in V$.
    \item If for some $i,j\in[n]$ with $i<j$ we have $E_{ij}\in V$, then $E_{kl}\in V$ for all $k,l\in[n]$ with $k\le i$ and $l\ge j$.
        {If we have $E_{ij}\in V$ for some $i,j\in[n]$ with $i>j$, then  the commutator $[E_{pq},E_{ij}]$ belongs to $V$ for all $p,q\in[n]$ with $p<q$.}
    \item If $V$ contains a matrix of the form $A=\begin{pmatrix}
                                                  \alpha & {b}^{T} \\
                                                  {c} & D
                                                \end{pmatrix}$ in the block partition determined by dimensions $(1,n-1)$, then it also contains matrices $\begin{pmatrix}
                                                  0 & {b}^{T} \\
                                                  0 & 0
                                                \end{pmatrix}$ and $\begin{pmatrix}
                                                  0 & 0 \\
                                                  {c} & 0
                                                \end{pmatrix}$.
    \item Claim (c) remains valid if we replace the first column and row by the $i$-th column an row for any $i\in[n]$.
    \item If for some $A\in V$ and $i,j\in[n]$ with $i\ne j$ we have $a_{ij}\ne0$, then $E_{ij}\in V$.
    \item If for some $A\in V$ and $i,j\in[n]$ with $i< j$ we have $a_{ii}\ne a_{jj}$, then $E_{ij}\in V$.
  \end{enumerate}
\end{lemma}

\begin{proof}
  Observe that $P=I+\lambda E_{ij}$ is upper triangular for any $\lambda\in\mathbb{C}$ with $P^{-1}=I-\lambda E_{ij}$, so that by assumption 
  we have $V\ni PAP^{-1}= A-\lambda[A,E_{ij}]- \lambda^2 E_{ij}A E_{ij}$. Since $\lambda$ was arbitrary in this computation and $A\in V$, we conclude first that $[A,E_{ij}]\in V$ and therefore $V\ni E_{ij}A E_{ij}=a_{ji}E_{ij}$, so that (a) follows. To get the first part of (b) assume first that $k<i<j$ and $l=j$. Replace in the above considerations $A$ by $E_{ij}$ and $E_{ij}$ by $E_{ki}$ to determine the desired fact $V\ni[E_{ij},E_{ki}]= -E_{kj}$.
  Once we know that, we get the rest from the case $i=k$, $i<j<l$, after replacing $A$ by $E_{ij}$ and $E_{ij}$ by $E_{jl}$. To get the second part of (b) replace in the above considerations $A$ by $E_{ij}$ and $E_{ij}$ by $E_{pq}$. We may choose various $p$ and $q$. For $p=j$ and $i>q>j$ we get $E_{iq}\in V$, for $j<p<i$ and $q=i$ we get $E_{pj}\in V$, and for $p=j,q=i$ we get $E_{ii}-E_{jj}\in V$.
  In the proof of (c) we apply similarity given by matrix $P=\mathrm{Diag}(\lambda,I)$ for arbitrary $\lambda\in\mathbb{C},\lambda\ne0$. A short computation reveals that $PAP^{-1}= \begin{pmatrix}
    \alpha & \lambda{b}^{T} \\
    \dfrac{1}{\lambda}{c} & D
  \end{pmatrix}$. Subtract this matrix from $A$ and divide by $1-\lambda$ to get $\begin{pmatrix}
    0 & {b}^{T} \\
    -\dfrac{1}{\lambda}{c} & 0
  \end{pmatrix}\in V$. When $\lambda$ tends to infinity, we get the first one of the desired matrices, when we multiply by $-\lambda$ and send $\lambda$ to zero, we get the second one. The proof of (d) is similar to that of (c), except that in $P$ we place $\lambda$ at the $i$-th position (instead of the first) on the diagonal. Claim (e) follows by applying (d) twice, once on the $i$-th position and keeping the row and once on the $j$-th position and keeping the column. To get (f) observe that the commutator of $A$ and $E_{ij}$ belongs to $V$ by the proof of (a). This commutator has $a_{ii}-a_{jj}$ on the $(i,j)$-th position by a short computation, and this is nonzero by assumption. The claim now follows by (e).
\end{proof}

\begin{prop}\label{pr:key}
   {If $k\ge1$, then t}he vector space $V$ of Lemma \ref{borel} has an invariant subspace $\mathcal{L}_s$ which is the span of $\{e_1,\ldots,e_s\}$ for some $s, 1\leqslant s\leqslant k$.
\end{prop}

\begin{proof}
  Assume the contrary. It then follows easily that for every $i,1\le i\le k$, there exist indices $l\le i$ and $j>i$ and an $A\in V$ such that $a_{j,l} \ne0$, so that $E_{j,l}\in V$ by Lemma \ref{Eij}(e). Using Lemma \ref{Eij}(b) we conclude that $E_{i+1,i}\in V$. It then follows by Lemma \ref{Eij}(a) that for every $i\le k$ 
  we have $E_{i,i+1}\in V$. A simple reflection on these two observations leads to existence of matrices $A,B\in V$ whose $(k+1)\times(k+1)$ northwest corners are of the respective forms
  \[
    \begin{pmatrix}
      0  &   &   &   \\
      \lambda_1 & \ddots  &   &   \\
        & \ddots &   &   \\
        &   & \lambda_k & 0
    \end{pmatrix}
    \quad\mbox{and}\quad\begin{pmatrix}
                           0 & \mu_1 &   &   \\
                            &   & \ddots &   \\
                            &   & \ddots  & \mu_k \\
                            &   &   & 0
                        \end{pmatrix},
  \]
  where the scalars $\lambda_i$ and $\mu_i$ can be chosen arbitrarily and all the other entries of these matrices are zero. A short computation reveals that \[[A,B]= \mathrm{Diag}(-\lambda_1\mu_1,\lambda_1\mu_1-\lambda_2\mu_2,\ldots,\lambda_k\mu_k, 0,\cdots,0)\] which can easily be made of rank strictly greater than $k$ with an appropriate choice of scalars. This contradicts assumption (A) of Lemma \ref{borel}.
\end{proof}

\begin{prop}\label{pr:double}
{  Let $k$ be an arbitrary {nonnegative} integer with $0\le k\le n-1$, and let $V$ be a linear space of $n\times n$ matrices such that $[A,B]$ is of rank at most $k$ for all $A,B\in V$, and $V$ is invariant under conjugation by all invertible upper triangular matrices. Then $\dim V\le nk+\lfloor \frac{(n-k)^2}{4}\rfloor +1$.
  Moreover, in the case of equality $V$ or $V^T$ is similar to the space of all matrices of the form $\begin{pmatrix}
     A&B\\0&C
   \end{pmatrix}$ where $A\in M_k(\mathbb{C})$, {$B\in M_{k\times (n-k)}(\mathbb{C})$, and} $C\in M_{n-k}(\mathbb{C})$ is of the form $C=\begin{pmatrix}
                                                                         \lambda I&D\\0&\lambda I
                                                                       \end{pmatrix}$ with $D\in M_{l\times (n-k-l)}(\mathbb{C})$ and $l\in \left\{ \lfloor \frac{n-k}{2}\rfloor ,\lceil \frac{n-k}{2}\rceil\right\}$.}
\end{prop}

{
Observe that exceptional cases for matrix $C$ have been ruled out by additional assumption that $V$ is invariant under conjugation by all invertible upper triangular matrices. }

\begin{proof}
  By Theorem \ref{maksrankone} we may assume that $k>1$.
  We will prove this proposition by induction on $n$. We know it for $n\le3$ by Theorems \ref{thm:2} and \ref{maksrankone}. Choose an $n$ and assume inductively that for any $n'<n$ and any $k', 0\le k'\le n'-1$, the proposition holds. 
  We assume 
  there is a vector space $V$ and an integer {$k, 2\le k\le n-1$}, such that  $[A,B]$ is of rank at most $k$ for each $A,B\in V$ and $\dim V= m\ge nk+\lfloor \frac{(n-k)^2}{4}\rfloor +1$ and $PVP^{-1}=V$ for every invertible upper triangular matrix $P$.
    By Proposition \ref{pr:key} space $V$ has an invariant subspace $\mathcal{L}_s$ of the form $\mathrm{span}\{e_1,\ldots.e_s\}$ for some $s, 1\leqslant s\leqslant k$. Choose $s$ to be the smallest possible such integer, so that for any $i<s$ there exists an $A\in V$ with $a_{i+1,i}\ne0$. Indeed, there exists an $A\in V$ such that $a_{jl}\ne0$ for some $j>i$ and $l\le i$, so that we can use Lemma \ref{Eij}(b) to get the desired claim.

  Consider first the case $s\ge2$. Any $A\in V$ has a block matrix representation $A=\begin{pmatrix}
                           A_1 & A_2 \\
                           0 & A_3
                         \end{pmatrix}$ with respect to this invariant subspace. Fix two members of $V$,
  $A_0=\begin{pmatrix}
                           A_1^0 & 0\\
                           0 & 0
                         \end{pmatrix}$ and
  $B_0=\begin{pmatrix}
                           B_1^0 & 0 \\
                           0 & 0
                         \end{pmatrix}$ chosen as in  the proof of Proposition \ref{pr:key}, where $k+1$ is replaced by $s$, so that the commutator $[A_1^0,B_1^0]$ is of rank $s$.
  Denote
  \[
  U=\left\{A_3\in M_{n-s}(\mathbb{C})\,:\ \mbox{there exist}\ A_1,A_2\ \mbox{with}\ A=\begin{pmatrix}
                           A_1 & A_2 \\
                           0 & A_3
                         \end{pmatrix}\in V \right\}
  \]
  and let $\pi:V\longrightarrow U,$ $A \mapsto A_3$. Since $\dim\ker\pi\le ns$ we get $\dim U\ge m-ns$, so that $\dim U\ge n(k-s)+\lfloor \frac{(n-k)^2}{4}\rfloor +1$.

  We want to show that for every $A_3,B_3\in U$ we have $\mathrm{rank}[A_3,B_3]\le k-s$. To this end we complete the two matrices up to $A=\begin{pmatrix}
                           A_1 & A_2 \\
                           0 & A_3
                         \end{pmatrix}$ and $B=\begin{pmatrix}
                           B_1 & B_2 \\
                           0 & B_3
                         \end{pmatrix}$ as members of $V$.
Choose an arbitrary $\lambda$ and observe that
\[
    \mathrm{rank}\left[\begin{pmatrix}
                 A_1+\lambda A_1^0 & A_2 \\
                 0 & A_3
               \end{pmatrix},\begin{pmatrix}
                 B_1+\lambda B_1^0 & B_2 \\
                 0 & B_3
               \end{pmatrix} \right]\le k
\]
because the matrices in the above commutator are members of $V$. Now, the $(1,1)$ block-entry of this commutator equals
\[
    [A_1,B_1]+\lambda[A_1,B_1^0]+\lambda[A_1^0,B_1]+\lambda^2[A_1^0,B_1^0]
\]
which is of rank $s$ for all but a finite number of $\lambda$'s, so that the rank of $[A_3,B_3]$ is no greater than $k-s$ as desired. Finally, $U$ is a linear space of $(n-s)\times(n-s)$ matrices of dimension greater than or equal to
\[
    n(k-s)+\left\lfloor\dfrac{(n-k)^2}{4}\right\rfloor+1
\]
such that for any pair $A_3,B_3\in U$ the rank of their commutator is no greater than $k-s$.
It is also clear that $U$ is invariant for conjugation by invertible upper triangular matrices.
So, by inductive hypothesis $$\dim U\le (n-s)(k-s)+\left\lfloor\dfrac{((n-s)-(k-s))^2}{4} \right\rfloor+1$$ which is no greater than the expression displayed above
. This shows that $s=k$ and the dimension of the kernel of $\pi$ is equal to $nk$, and that $$\dim U= \left\lfloor\dfrac{(n-k)^2}{4}\right\rfloor+1.$$ Using the result of Schur for $U$ we get the desired form of matrices.

It remains to consider the case when $\mathcal{L}_1$ is an invariant subspace of every member of $V$.
Modulo scalar matrices, we may assume that elements of $V$ have the $(1,1)$ entry equal to zero. So, they have the first column equal to zero and we may write them as $A=\begin{pmatrix}
                                                                 0 & a^\mathrm{T} \\
                                                                 0 & A'
                                                               \end{pmatrix}$ with respect to the obvious block partition.
Write a member $B$ of $V$ with respect to the same partition. 
If the rank of $[A',B']$ is always no greater than $k-1$, then we conclude by the inductive hypothesis that $\dim V\le 1 +(n-1)+1+(k-1)(n-1)+ \left\lfloor \dfrac{(n-1-k+1)^2}{4}\right\rfloor < 1+nk+\left\lfloor\dfrac{(n-k)^2}{4} \right\rfloor$ contradicting the starting assumption
.

We have shown that there exist $A_0$ and $B_0$ in $V$ such that the rank of $[A_0', B_0']$ equals $k$. Recall that we have subtracted the scalar matrices to assume without loss that the $(1,1)$ entry of all matrices from $V$ is zero. Denote by $t\ge1$ the highest possible number of zeros in the beginning of the first row of all members of $V$. It follows by Lemma \ref{Eij}(b),(e)\&(f) that $t$ equals the (maximal) number of the first zero columns shared by all elements of $V$. So, rewrite elements of $V$ with respect to the block partition given by respective sizes $1,t-1,n-t$:

\noindent
(*)$\quad\quad\quad\quad  \quad\quad  \quad\quad  \quad\quad    A=\begin{pmatrix}
        0 & 0 & a_1^\mathrm{T} \\
        0 & 0 & A_1 \\
        0 & 0 & A_2
      \end{pmatrix}
$

\noindent
and observe by Lemma \ref{Eij}(b) that vector $a_1$ can attain any element of $\mathbb{C}^{n-t}$ independently of the choice of $A_1$ and $A_2$. Next, we want to show that the rank of any submatrix $A_2$ in Block Representation 
(*) is at most $k$. To this end we assume the contrary and then assume with no loss that the matrix $A_0$ chosen above also satisfies rank$\,A_2^0\ge k+1$ and $a_1^0=0$. Now, we can and will choose a matrix $B$ which will lead to a contradiction. Indeed, matrix $A_0'$ has greater rank than matrix $[A_0',B_0']$ so that there is a $b_1\in \mathbb{C}^{n-t}$ such that $\begin{pmatrix}
                               0 \\
                               (A_2^0)^\mathrm{T}b_1
                             \end{pmatrix}$ is not contained in the image of $[A_0',B_0']^\mathrm{T}$. Now, let $B=\begin{pmatrix}
        0 & 0 & b_1^\mathrm{T} \\
        0 & 0 & B_1^0 \\
        0 & 0 & B_2^0
      \end{pmatrix}$, which clearly belongs to $V$. Then
\[
    [A_0,B]=\begin{pmatrix}
            0 & 0 & -b_1^\mathrm{T}A_2^0 \\
            0 & 0 & A_1^0B_2^0-B_1^0A_2^0 \\
            0 & 0 & [A_2^0,B_2^0]
          \end{pmatrix}
\]
has rank no smaller than $k+1$ after observing that the southeast $2\times2$ corner equals $[A_0',B_0']$; but this conclusion is contradicting the starting assumption (A) on $V$. So, we determined that rank$\,A_2\leq k$ for all $A\in V$. Finally, recall an old result of Flanders \cite[Theorem 1]{Flan} giving us the upper bound $k(n-t)$ for dimension of the southeast block corners of matrices represented in (*)
, while the dimension of the northeast pair of blocks is clearly bounded by $t(n-t)$. Taking also into account the scalars, we get \[\dim V\le 1+(t+k)(n-t)=1+nk-\left(t- \dfrac{n-k}{2}\right)^2+\dfrac{(n-k)^2}{4}\le 1+nk+\left\lfloor\dfrac{(n-k)^2}{4} \right\rfloor.\]
This shows that the equality holds in the starting {assumption on $m$}
. 
The equality above holds only when $t$ is either $\left\lfloor\dfrac{n-k}{2} \right\rfloor$ or  $\left\lceil\dfrac{n-k}{2} \right\rceil$, and where $A_1$ and $a_1$ are arbitrary, and the dimension of the space of matrices $A_2$ in the block representation (*) is $k(n-t)$. It follows easily from \cite[Theorem 2]{Flan}, that up to similarity the space of all possible matrices $A_2$ consists of either all matrices with last $n-t-k$ rows zero or all matrices with first $n-t-k$ columns zero. In the second case we get (one form of) the desired structure of matrices. However, in the first case there exist matrices $A$ and $B$ in $V$ such that the rank of their commutator is greater than $k$, contradicting the assumptions of the proposition.
\end{proof}

The following is the main result of the paper which verifies the first part of Conjecture \ref{conj}.

\begin{thm}\label{th:main}
  Let $k$ be an arbitrary integer with $0\le k\le n-1$, and let $V$ be a linear space of $n\times n$ matrices such that $[A,B]$ is of rank at most $k$ for all $A,B\in V$. Then $\dim V\le nk+\lfloor \frac{(n-k)^2}{4}\rfloor +1$.
\end{thm}

\begin{proof}
  Towards a contradiction assume that  $m=\dim V> nk+\lfloor \frac{(n-k)^2}{4}\rfloor +1$. So, the variety of all $m$-dimensional subspaces of $M_n(\mathbb{C})$ satisfying conditions (A) and (B) is nonempty. By Lemma \ref{borel} there exists an $m$-dimensional subspace $V_0$ satisfying conditions (A) and (B) such that $P^{-1}AP\in V_0$ for all $A\in V_0$ and all invertible upper triangular matrices $P$, in contradiction with Proposition \ref{pr:double}.
\end{proof}

\section{The case of algebras
}

In this section we prove the second part of Conjecture \ref{conj} under {the additional assumption} that $V$ is an algebra.

\begin{thm}
  Let $k<n$ and $\mathcal{A}\subseteq M_n(\mathbb{C})$ be an algebra such that $\mathrm{rank}\,(AB-BA)\le k$ for all $A,B\in\mathcal{A}$, and that $\displaystyle\dim\mathcal{A}=nk +\left\lfloor\dfrac{(n-k)^2}{4}\right\rfloor+1$. Then $\mathcal{A}$ is of the form described in Conjecture \ref{conj}.
\end{thm}

\begin{proof} {
  By Theorem \ref{maksrankone} we may assume that $k>1$.
  We first observe that $\mathcal{A}$ cannot be all of $M_n(\mathbb{C})$, so that it must be reducible by Burnside's Theorem. Moreover, there exists a block triangularization of $\mathcal{A}$ with irreducible diagonal blocks. Let $n_1,n_2,\ldots ,n_r$ be the sizes of the diagonal blocks. We first adjust the ideas from Section \ref{sec:main}. Let us adjoin to the conditions (A) and (B) also the condition
\begin{enumerate}[(C)]
  \item Linear space $V$ is made of block upper triangular matrices corresponding to the partition $(n_1,n_2,\ldots ,n_r)$.
\end{enumerate}
The set $X'$ satisfying this additional condition is nonempty since it contains $\mathcal{A}$. Clearly, analogues of Lemmas \ref{lem:7} and \ref{borel} are still valid for $X'$. This means, in particular, that $X'$ is a projective variety and that there exists a subspace $V_0$ of $M_n(\mathbb{C})$ of dimension $nk +\left\lfloor\frac{(n-k)^2}{4}\right\rfloor+1$ satisfying Conditions (A), (B), and (C) such that $P^{-1}AP\in V_0$ for all $A\in V_0$ and all invertible upper triangular matrices $P$. According to Proposition \ref{pr:double} the space $V_0$ is of the (non-exceptional) form described in Conjecture \ref{conj}. However, the partition corresponding to the block triangularization of such a space has the first or last part no smaller than k. Due to irreducibility of blocks in $\mathcal{A}$ we therefore have either $n_1=k$ or $n_r=k$. We may assume that $n_1=k$, since otherwise we consider the algebra of the transposes of matrices from $\mathcal{A}$.

So, in the block matrix notation corresponding to the partition $(k,n-k)$ every member of $\mathcal{A}$ is of the form $A=\begin{pmatrix}
  A_1 & A_2 \\
  0 & A_3
  \end{pmatrix}$ and an arbitrary $A_1\in M_{k}(\mathbb{C})$ can appear in the northwest corner. Choose another member $B=\begin{pmatrix}
  B_1 & B_2 \\
  0 & B_3
  \end{pmatrix}\in\mathcal{A}$ and note that
  \[
    \mathrm{rank}[A,B]\ge \mathrm{rank}[A_1,B_1]+\mathrm{rank}[A_3,B_3].
  \]
  Recall that $k>1$. The first term on the right hand side of this inequality can always achieve the value $k$ leading to the conclusion that $[A_3,B_3]=0$. Indeed, choose $A_0\in\mathcal{A}$ (with corresponding blocks denoted by $A_1^0,A_2^0,$ and $A_3^0$) and $B_0\in\mathcal{A}$ (with corresponding blocks denoted by $B_1^0,B_2^0,$ and $B_3^0$), such that $\mathrm{rank}[A_1^0,B_1^0]=k$. Given $A$ and $B$ as before, we have that $\mathrm{rank}[A_1+\lambda A_1^0, B_1+\lambda B_1^0]=k$ for all but a finite number of $\lambda\in\mathbb{C}$. So, after extending the estimate above,
  \[
    \mathrm{rank}[A+\lambda A_0,B+\lambda B_0]\ge \mathrm{rank}[A_1+\lambda A_1^0,B_1+\lambda B_1^0]+\mathrm{rank}[A_3+\lambda A_3^0,B_3+\lambda B_3^0],
  \]
  we observe that $[A_3+\lambda A_3^0,B_3+\lambda B_3^0]=0$ for all but a finite number of $\lambda\in\mathbb{C}$, thus leading to the desired conclusion. The dimension of $\mathcal{A}$ is no greater than the sum $S$ of $nk$ and the dimension of the southeast block. We estimate further the dimension of the southeast block by Schur's result. So, $S$ is no greater than
  \[
    nk+\left\lfloor\dfrac{(n-k)^2}{4}\right\rfloor+1.
  \]
  On the other hand we have that $\displaystyle\dim\mathcal{A}=nk +\left\lfloor\dfrac{(n-k)^2}{4}\right\rfloor+1$ which implies that 
 for $A=\begin{pmatrix}
  A_1 & A_2 \\
  0 & A_3
  \end{pmatrix}$ the matrices $A_1$ and $A_2$ can be taken arbitrarily, while in the southeast corner any two matrices have a zero commutator. Using the Schur's result again we conclude that the southeast corner is of the form described in the introduction. In details, either $\mathcal{A}$ belongs to exceptional cases, or we have for every $A\in\mathcal{A}$ that
  \[
    A=\begin{pmatrix}
        * & * & * \\
        0 & \lambda I & * \\
        0 & 0 & \lambda I
      \end{pmatrix},
  \]
  with the block partition obeying the respective dimensions $k$, $r_k$, and $s_k$, where $r_k$ and $s_k$ are equal to $\left\lfloor\dfrac{n-k}{2}\right\rfloor$ and $\left\lfloor \dfrac{n-k+1}{2} \right\rfloor$ in either order.
  }
\end{proof}

\end{document}